\documentclass[12pt]{amsart}
\usepackage[applemac]{inputenc}
\usepackage[T1]{fontenc}

\usepackage{ulem}
\usepackage{amsmath,esint,color}
\usepackage{amsthm,latexsym,epsfig,graphicx,mathrsfs}
\usepackage{amsmath}
\usepackage{amsfonts}
\usepackage{amssymb}
\usepackage{fullpage}

\date{}
\definecolor{sah}{rgb}{0.66,0.33, 0.04}
\definecolor{adel4}{cmyk}{1,0,0,0}
\definecolor{adel3}{rgb}{0.66,0.33, 0.04}
\definecolor{adel1}{cmyk}{0,0.20,1,0}
\definecolor{adel2}{cmyk}{0,0.40,1,0.30}
\definecolor{adel0}{rgb}{0.99,0.60, 0.30}
\definecolor{trut}{rgb}{0.99,0.80, 0.00}
\definecolor{trus}{rgb}{0.00, 0.50, 0.00}
 \definecolor{trust}{rgb}{0.99, 0.99, 0.80}
\definecolor{MaCouleur}{rgb}{0,0.9,0.3}

\newcommand{\Z}{\mathbb{Z}}
\newcommand{\RR}{\mathbb{R}}
\newcommand{\T}{\mathbb{T}}

\theoremstyle{plain}

\newtheorem{theorem}{Theorem}

\newtheorem{remark}{Remark}

\title{Ergodicity effects on transport-diffusion equations with localized damping}

\begin{document}
\author[K. Ammari]{Ka\"is Ammari}
\address{UR Analysis and Control of PDEs, UR13ES64, Department of Mathematics, Faculty of Sciences of Monastir, University of Monastir, 5019 Monastir, Tunisia}
\email{kais.ammari@fsm.rnu.tn}
\author[T. Hmidi]{Taoufik Hmidi}
\address{Univ Rennes, CNRS, IRMAR -- UMR 6625, 
F-35000 Rennes, France}
\email{thmidi@univ-rennes1.fr}


\begin{abstract}
The main objective of this paper is to study the time decay of transport-diffusion equation with inhomogeneous localized damping in the multi-dimensional torus. The drift is governed by an autonomous Lipschitz vector field and the diffusion by the standard heat equation with small viscosity parameter $\nu$. In the first part we deal with the inviscid case and show some results on the time decay of the energy using in a crucial way the ergodicity and the unique ergodicity of   the flow generated by the drift. In the second part we analyze the same problem with small viscosity and provide quite similar results on the exponential decay uniformly with respect to the viscosity in some  logarithmic time scaling  of the \mbox{type $t\in [0,C_0\ln(1/\nu)]$}.    
\end{abstract}

\subjclass[2010]{35B40}
\keywords{Ergodicity effects, transport-diffusion equations, localized damping}

\maketitle

\tableofcontents

\section{Introduction}
In this short note we aim at exploring the time decay of the energy for 
 the   transport-diffusion  equation with localized damping   in the multi-dimensional torus $\T^d,$
\begin{equation} \label{eqn:tran}
\left\{ \begin{array}{ll}
  \partial_t \theta(t,x) + v(x) \cdot \nabla \theta(t,x) -\nu \Delta \theta(t,x)= -\phi(x)\,\theta(t,x) \quad (t,x)\in \mathbb{R}_+\times \T^d,&\\ 
  \theta(0,x)=\theta_0(x),
  \end{array}\right.
\end{equation}
where  $\phi: \T^d\to \RR_+$ is a given time-independent positive function and $\nu\geq0$ is a viscosity parameter. The vector field $v:\T^d\to \RR^d$ is assumed to be autonomous,  solenoidal, that is $\textnormal{div} \, v=0$,  and  belongs to the Lipschitz class. 
When $\phi$ is  bounded away from zero, meaning that one may find a constant $\mu>0$ such that $\phi(x)\geq \mu$ almost everywhere, then performing $L^2$-energy estimate it is quite easy to get 
\begin{align}\label{Energy-es}
\nonumber \frac12\frac{d}{dt}\|\theta(t)\|_{L^2}^2+\nu\|\nabla \theta(t)\|_{L^2}^2&=-\int_{\T^d}\phi(x)|\theta(t,x)|^2dx\\
&\leq-\mu  \|\theta(t)\|_{L^2}^2.
\end{align}
This implies in particular  the following time decay
$$
\forall\, t \geq 0,\quad \left\|\theta(t)\right\|_{L^2} \le 
\|\theta_0\|_{L^2} e^{-\mu t}.
$$
The main task here is to investigate the time decay when $\phi$ is not uniformly distributed on the torus and its support may be localized in a small  domain. In the absence of the drift $v=0$ and the viscosity $\nu=0$ the stabilization can not occur because in this case the solution is explicitly recovered through the formula,
$$
\theta(t,x)=\theta_0(x) e^{-t\phi(x)}.
$$
Thus the energy is conserved for  initial data whose supports do not intersect the support \mbox{of $\phi$.}

\medskip

Now let us discuss and  specify the main feature of the full equation \eqref{eqn:tran} which is mainly   governed  by two mechanisms,  the advection represented by  the velocity field $v$ and the diffusion given by the viscous part. First, let us neglect the diffusion part and focus  on  the inviscid equation \eqref{eqn:tran} with the vanishing viscosity $\nu=0$, then  one can  guess  that  the transport structure with special vector fields  would allow to advect  the mass in a recurrent way to the activity  zone of the damping   that could  guarantee the time decay.  Thus  we expect the ergodicity to be a crucial  phenomenon for flowing and spreading  the mass and de facto for  stabilizing  the inviscid system. As we shall see in  Theorem \ref{Theo-invic}, we will make  this intuitive idea more pertinent and rigorous in the framework of ergodic flows.  We emphasize that in our case the ergodicity allows to get  the  time  decay, however for the exponential decay we need  to deal with a special class of ergodic flows, known in the literature by  the  uniquely ergodic flows.  We shall give later some details about this terminology. Notice that  in the periodic setting lot of results around the ergodicity of continuous and  discrete dynamical system are  well-developed in the literature, we can refer for instance to \cite{Mao, Nikola, Saito, Walters}. 

\medskip

In \cite{Lebeau} Lebeau characterized the value of the best decay rate for the damped wave equation in terms of two quantities: the spectral abscissa and the mean value of the damping coefficient along the rays of geometrical optics, i.e., the Birkhoff limit of the damping coefficient, which illustrate the ergodicity effect on the dynamics. In other hand in \cite{Sjostrand} Sj\"ostrand give an asymptotic distribution of the eigenfrequencies associated to the damped wave operator, and show that the spectra is confined to a band determined by the Birkhoff limits of the damping coefficient and that certain averages of the imaginary parts converge to the average of the damping coefficient. Later and based on the work of Sj\"ostrand \cite{Sjostrand}, Asch and Lebeau \cite{Asch} give an asymptotic expression for the distribution of the imaginary parts of the eigenvalues for a radially symmetric geometry and establish the close links between the eigenfunctions and the rays of geometrical optics. For the one dimensional models, we can see \cite{Freitas} and \cite{Cox}. 

\medskip

Now let us focus on  the second mechanism of the equation \eqref{eqn:tran} and analyze the viscosity effects. We point out that without the  zero-mean   condition there is no hope to get a time decay, to be convinced it suffices to work with $\theta_0$ as a constant function   and $\phi=0$ then the solution is a constant.  However,  the spatial average $\int_{\T^d}\theta(t,x)dx$ is not conserved in general  due to the inhomogeneous damping term. Nevertheless, when $\phi$ is even $\phi(-x)=\phi(x)$ then odd symmetry is preserved by the dynamics. Consequently if we impose to  $\theta_0$ to be  odd then the solution keeps this symmetry for any time  and therefore the spatial average is always zero. With this special symmetry one may use Poincar\'e inequality in the torus and deduce from \eqref{Energy-es} that  
$$
 \frac12\frac{d}{dt}\|\theta(t)\|_{L^2}^2+\nu\|\theta(t)\|_{L^2}^2\leq0.
$$
From this we infer the exponential decay: $\forall\, t\geq0,\quad \|\theta(t)\|_{L^2}\le \|\theta_0\|_{L^2} e^{-\nu t}$, which is of course not uniform for vanishing viscosity $\nu\to0$. It is worthy to point out  that without the zero  average, the exponential decay is no longer satisfied. In this case  the transfer of mass to low frequencies  could happen preventing the energy from  any  decay rate. Now what happens   if we combine   both mechanisms and deal with the full  equation \eqref{eqn:tran}. In this setting,  several questions of important  interest emerge. For instance, how  the ergodicity  is  affected by  the diffusion and in particular what about the recurrence behavior,  which is essential to guide the trajectories to  the damping zone ?  Is it altered or reinforced by the diffusion? Another connected problem is whether or not the time decay occurs uniformly for the vanishing  viscosity. It is so hard to give a full and a complete answer to those problems. Here we shall only provide a partial answer by looking to the regime where the ergodicity is not altered by the diffusion. More precisely, we are able to exhibit  a  logarithmic  time scaling $0\leq t\leq \log\frac1\nu$ where  the exponential  time decay continue to survive  uniformly  with respect to small  viscosity. A precise statement will be established in Theorem \ref{Theo-diff}. 

\medskip

In order to  formulate carefully our main results, we  need to recall some tools and fix some terminology from the theory of  dynamical systems.  First, let us  briefly see how to define the flow on the torus associated to a given autonomous  vector field $ {v}:\T^d\to\RR^d$. Here the multi-dimensional torus $\T^d$ is identified to the factor space $ \RR^d/\Z^d$ and any  function defined on the torus can be identified  with  its lift defined on  $\RR^d.$  Now  let $x$ be an arbitrary point of the set  $ \T^d$ and  define the orbit $t\mapsto X(t)\in\RR^d$ as the  unique  solution of the ODE$$
\dot{X}(t)=v(X(t)),\, X(0)=x.
$$
By projecting down  this orbit on the torus using  the natural projection $\pi: \RR^d\to \RR^d/\Z^d$, allows to cover the  analogous  orbit on the torus. By this way we can generate a flow on the torus that we denote   throughout this  paper by $\psi$.  To avoid the  difficulties related to the blow-up phenomenon, we assume that the vector field is complete meaning that all the orbits  are globally defined in time.  A sub-class of those vector fields  is given by Lipschitz class which will be a canonical assumption in the paper.  Therefore the flow $\psi$ can be viewed  as a mapping $\psi:\, \RR\times\T^d\to \T^d$ and still satisfies the ODE
\begin{equation}\label{flow1} 
\left\{ \begin{array}{ll}
  \partial_t\psi(t,x)=v(\psi(t,x))\quad (t,x)\in \mathbb{R}_+\times \T^d,&\\ 
  \psi(0,x)=x\in\T^d.
  \end{array}\right.
\end{equation}
Then by this way  we generate  a group of continuous transformation on the torus $\big\{\psi_t\big\}_{t\in\RR}$, and in particular we have  $\psi_t\,:\T^d\to\T^d$ with  $\psi_t\circ\psi_s=\psi_{t+s}$.  Moreover, since the vector field $v$ is assumed to be solenoidal or  incompressible then the  Lebesgue measure denoted by $\lambda$ or $dx$ is invariant by the flow. This means that for any Borel set $A\subset \T^d$ we have
$$
\forall t\in\RR,\quad \lambda(\psi_t(A))=\lambda(A).
$$
Remark that the torus $\T^d$ is normalized, that is $\lambda(\T^d)=1.$ A set $A\subset \T^d$ is called  {\it invariant set}  if  $\psi_t(A)=A, \forall t\in\RR$. The one-parameter flow $\big\{\psi_t\big\}_{t\in\RR}$ is said to be {\it Ergodic} with respect to Lebesgue measure if any invariant set has  either zero measure or full measure.  This  flow is said {\it uniquely ergodic} if it admits  only one invariant Borel probability measure, which is necessary the Lebesgue measure. A useful tool in our study is the so-called  Birkhooff's Ergodic Theorem. Before giving a precise statement we need  to  introduce for  $p\in[1,\infty)$, the standard Lebesgue space $L^p(\T^d)$ which is  the set of measurable functions $f:\T^d\to \RR$ such that
 $$
 \|f\|_{p}=\left(\int_{\T} |f(x)|^p dx\right)^{\frac1p}<\infty,
 $$
 with the usual adaptation for the case $p=+\infty$. We also denote by $\langle f \rangle$ the spatial average  of   $f$  by
 $$
 \langle f \rangle=\int_{\T^d}f(x)dx.
 $$
 The following theorem where we collect several results  in  classical dynamical systems   can be found in standard books dealing with Ergodic Theory. For instance, we refer to  \cite{Petersen,Walters} and the references therein.
\begin{theorem}\label{Birk-Ergod}
Let $\phi \in L^1(\T^d)$ and $\{\psi_t\}_{t\in\RR}$ be a one-parameter group of transformations on the torus $\T^d$ preserving Lebesgue measure. Then there exists $\phi_\star\in  L^1(\T^d)$ such that

\begin{equation}\label{Birk-Ergod1}
\lim_{t\to+\infty}\frac1t\int_0^t\phi(\psi(\tau,x))d\tau=\phi_\star(x), a.e.
\end{equation}
where $\phi_\star$ satisfies the properties
\begin{enumerate}
\item For any $t\in\RR, \phi_\star(\psi(t,x))=\phi_\star(x),$ a.e.
\item If $\phi \in L^p(\T^d)$ then $\phi_\star\in L^p(\T^d)$ with $ \|\phi_\star\|_{L^p}\leq \|\phi\|_{L^p}.$
\item If $A$ is invariant by the flow then 
$$
\int_A \phi_\star(x)dx=\int_A\phi(x) dx.
$$
\item If $\phi\in L^p(\T^d)$ with   $p\in[1,\infty)$, then the convergence in \eqref{Birk-Ergod1} holds in $L^p(\T^d)$.
\item If the flow $\{\psi_t\}_{t\in\RR}$ is ergodic then $\phi_\star(x)=  \langle \phi \rangle$ a.e.
\item If the flow $\{\psi_t\}_{t\in\RR}$ is uniquely ergodic and  $\phi\in \mathscr{C}(\T^d;\RR)$, then the  convergence in \eqref{Birk-Ergod} to the spatial average  $ \langle \phi \rangle$ occurs everywhere and   is uniform.
\end{enumerate}
\end{theorem}

%
Next we shall give some specific vector fields that generate in the $2d$-torus  unique ergodic flows. \\
{\it Examples of uniquely ergodic flows on the torus}. The following result comes from Sait\^o \cite{Saito} and Maoan \cite{Mao}. Take a $\mathscr{C}^1$ solenoidal vector field $v=(v^1,v^2)$ without any zero and  such that 
$$
v^1(x,y)=\sum_{n,m\in\Z} a_{mn}e^{2i\pi(mx+ny)},\quad v^2(x,y)=\sum_{n,m\in\Z} b_{mn}e^{2i\pi(mx+ny)},
$$
Then the flow associated to $v$ is ergodic if and only if 
$$
a_{00} b_{00}\neq0 \quad \hbox{and}\quad \frac{a_{00}}{b_{00}}\notin \mathbb{Q}.
$$
Under the same conditions the flow is also uniquely ergodic.
\section{Inviscid case}
This section is devoted  the case $\nu=0$ where  the equation \eqref{eqn:tran} reduces to 
\begin{equation} \label{eqn:tran1}
\left\{ \begin{array}{ll}
  \partial_t \theta(t,x) + v(x)\cdot \nabla \theta(t,x)= -\,\phi(x)\, \theta(t,x)\quad (t,x)\in \mathbb{R}_+\times \T^d,&\\ 
  \theta_0|_{t=0}=\theta_0,
  \end{array}\right.
\end{equation}
Our main result deals with the time decay of the energy. 
\begin{theorem}\label{Theo-invic}
Let $d\geq1$, $v:\T^d\to\RR^d$ be in the Lipschitz class and $ \phi:\T^d\to\RR$ be a positive and integrable function such that  its average $ \langle \phi \rangle>0$. Then the following assertions hold true.  
\begin{enumerate}
\item If  the flow $\{\psi_t\}_{t\in\RR}$ is ergodic, then for $ \theta_0\in L^p(\T^d)$ with $ p\in[1,\infty)$,  we have 
$$
\lim_{t\to+\infty}\|\theta(t)\|_{L^p}=0.
$$
\item  If  the flow $\{\psi_t\}_{t\in\RR}$ is uniquely  ergodic and  $\phi\in \mathscr{C}(\T^d)$. Then for any $ \theta_0\in L^p(\T^d)$ with $p\in[1,\infty]$  and  for any $\mu\in [0, \langle \phi\rangle)$ there exists $T_0>0$ such that  
$$
\forall t\geq T_0,\quad \|\theta(t)\|_{L^p}\leq \|\theta_0\|_{L^p}e^{-\mu t}.
$$
\item Assume that $\phi_\star(x)>0, a.e.$ and $ \theta_0\in L^p(\T^d)$ with $ p\in[1,\infty)$, then
$$
\lim_{t\to+\infty}\|\theta(t)\|_{L^p}=0.
$$
\end{enumerate}
\end{theorem}
\begin{remark}
Notice that in the point $(iii)$ we do not assume the ergodicity of the flow, which is substituted by the strict positivity of $\phi_\star$ almost everywhere. From this result, we can deduce that for any non negligible invariant set  $A$, we have necessary $\textnormal{supp } \phi\cap A$ is a non negligible set. Indeed, using Theorem \ref{Birk-Ergod}-$(iii)$ we achieve that
$$
\int_{A}\phi_\star(x) dx=\int_A\phi(x)dx.
$$
From the assumption on $\phi_\star$ we easily get $\int_A\phi(x)dx>0$ and this ensures the desired result. This gives an insight where  one should localize  the control region in order to get the time  decay.
\end{remark}
\begin{proof}
{\bf{(i)}}  Let $\{\psi_t\}_{t\in\RR}$ be the flow associated to the velocity field $v$. Then  
 by setting  $\eta(t,x)=\theta(t,\psi(t,x))$  it is quite easy to check that
$$
\partial_t\eta(t,x)=-\eta(t,x) \,\phi(\psi(t,x)).
$$
Thus we can recover from this ODE the solution as follows,
\begin{equation}\label{Repr}
\eta(t,x)=\theta_0(x) e^{-\int_0^t\phi(\psi(\tau,x))d\tau}.
\end{equation}
Since the flow  is assumed to be ergodic  then  according to Theorem \ref{Birk-Ergod}-(v) the function $\phi_\star$ given by \eqref{Birk-Ergod} is constant and coincides with its spatial average, meaning that, 
\begin{equation}\label{Birk1}
\lim_{t\to+\infty}\frac1t\int_0^t\phi(\psi(\tau,x))d\tau=\langle \phi\rangle,\quad  a.e.
\end{equation}
Now using the identity \eqref{Repr} combined with the invariance of the Lebesgue measure under the flow we deduce that
\begin{eqnarray}\label{Repr01}
\nonumber\|\theta(t)\|_{L^p}^p&=&\|\eta(t)\|_{L^p}^p\\
&=&\int_{\T^2}|\theta_0(x)|^p e^{-p\int_0^t\phi(\psi(\tau,x))d\tau} dx.
\end{eqnarray}
From  \eqref{Birk1} and  $\langle \phi\rangle >0$ we infer that  
$$
\lim_{t\to+\infty}\int_0^t\phi(\psi(\tau,x))d\tau=+\infty\quad  a.e.
$$
Hence we get the pointwise convergence
$$
\lim_{t\to+\infty} e^{-p\int_0^t\phi(\psi(\tau,x))d\tau}=0\quad  a.e.
$$
Moreover from the positivity of $\phi$ we get the uniform domination
$$
\forall t\geq0,\,|\theta_0(x)|^p e^{-p\int_0^t\phi(\psi(\tau,x))d\tau}\leq |\theta_0(x)|^p\quad  a.e.
$$
Thus, by virtue of  Lebesgue's dominated  convergence theorem we obtain for any $p\in[1,\infty),$
$$
\lim_{t\to+\infty} \|\theta(t)\|_{L^p}=0.
$$
{\bf{(ii)}} Since $\phi$ is continuous and the  flow $\{\psi_t\}_{t\in\RR}$ is supposed to be uniquely  ergodic  then using once again Theorem \ref{Birk-Ergod}-(vi) we get that   the convergence in \eqref{Birk1}  is uniform. Consequently for any $0<\varepsilon<\langle \phi\rangle$  we can find $T_0>0$ such that
$$
\forall t\geq T_0,\,\forall x\in \T^d,\quad \langle \phi\rangle-\varepsilon  \leq \frac1t\int_0^t\phi(\psi(\tau,x))d\tau
$$
Combined with \eqref{Repr01} we find that
\begin{equation*}
\forall t\geq T_0,\quad \|\theta(t)\|_{L^p}
\leq e^{-(\langle\phi\rangle-\varepsilon) t} \|\theta_0\|_{L^p}.
\end{equation*}
This achieves the proof by taking $\mu=\langle \phi\rangle-\varepsilon$.\\
{\bf{(iii)}}  By virtue of  Theorem \ref{Birk-Ergod} we have 
$$
\lim_{t\to+\infty}\frac1t\int_0^t\phi(\psi(\tau,x))d\tau=\phi_\star(x), a.e.
$$
Therefore, we get from  the assumption $\phi_\star(x)>0, a.e.$ that
$$
\lim_{t\to+\infty}\int_0^t\phi(\psi(\tau,x))d\tau=+\infty, a.e.
$$
Then it is enough to apply   Lebesgue's dominated  convergence theorem with \eqref{Repr01} in order to get the desired result.
\end{proof}
%
%
\section{Viscous case}
In this section we shall deal with the viscous case and explore the exponential decay for small viscosity. Our  main result reads as follows.
\begin{theorem}\label{Theo-diff}
Let $d\geq1, p\in[2,\infty)$ and  $ \phi:\T^d\to\RR$ be a positive Lipschitz function such that  $ \langle \phi \rangle>0$. 
Assume that  the flow $\{\psi_t\}_{t\in\RR}$ is uniquely  ergodic.  Let  $\mu\in [0, \langle \phi\rangle)$, then there exists a constant   $C_0>0$ depending explicitly  on the Lipschitz norms of $\phi$ and $v$ and a constant $C_p$ depending only on $p,\, \mu $ and $C_0$ such that for any  $\nu\in(0,1)$ the solution $\theta$ of \eqref{eqn:tran} satisfies 
$$\forall t\in[0, C_0\ln(1/\nu)],\quad \|\theta(t)\|_{L^p}\leq C_p\|\theta_0\|_{L^p}e^{-\mu t}.
$$
We emphasize that neither $C_0$ nor $C_p$ depends on the small viscosity  $\nu\in(0,1).$

\end{theorem}
\begin{proof}
Set $\eta(t,x)=\theta(t,\psi(t,x))$ then it satisfies the equation 
$$
\partial_t\eta(t,x)-\nu( \Delta\theta)(t,\psi(t,x))=-\eta(t,x) \,\phi(\psi(t,x)).
$$
Define $V(t,x)=\int_0^t\phi(\psi(\tau,x))d\tau$ and $h(t,x)=e^{V(t,x)}\eta(t,x)$ then it is easy to check 
$$
\partial_t h(t,x)-\nu e^{V(t,x)}\, (\Delta\theta)(t,\psi(t,x))=0.
$$
We multiply this equation by $|h|^{p-2}h$ and we integrate in space variable,
$$
\frac{1}{p}\frac{d}{dt}\|h(t)\|_{L^p}^p-\nu\int_{\T^d}e^{V(t,x)}( \Delta\theta)(t,\psi(t,x)) |h(t,x|^{p-2}h(t,x)dx=0.
$$
This energy identity  can be written in the form
\begin{equation}\label{Energ1}
\frac{1}{p}\frac{d}{dt}\|h(t)\|_{L^p}^p-\nu \,\mathcal{I}(t)=0,
\end{equation}
with 
$$
\mathcal{I}(t):=\int_{\T^d} e^{pV(t,x)}( \Delta\theta)(t,\psi(t,x)) \big|\theta\big(t,\psi(t,x)\big)\big|^{p-2}\theta\big(t,\psi(t,x)\big)dx.
$$
Since the flow preserves Lebesgue measure and $\psi^{-1}_t=\psi_{-t}$ then change of variables leads to
$$
\mathcal{I}(t)=\int_{\T^d} e^{pV(t,\psi(-t,x))}( \Delta\theta)(t,x) \big|\theta\big(t,x\big)\big|^{p-2}\theta\big(t,x\big)dx.
$$
Integrating by parts yields
\begin{align*}
\mathcal{I}(t)&=-(p-1)\int_{\T^d} e^{pV(t,\psi(-t,x))}|\nabla\theta(t,x)|^2 \big|\theta\big(t,x\big)\big|^{p-2}dx\\
&-p\int_{\T^d} e^{pV(t,\psi(-t,x))} \big|\theta\big(t,x\big)\big|^{p-2}\theta\big(t,x\big)\,\nabla \theta(t,x)\cdot \nabla (V(t, \psi(-t,x))dx.
\end{align*}
Using the inequality $|ab|\leq \varepsilon a^2+\frac{1}{4\varepsilon} b^2$, for any $\varepsilon>0$ with a suitable choice of $\varepsilon$ allows to get
 \begin{align*}
&p\int_{\T^d} e^{pV(t,\psi(-t,x))} \big|\theta\big(t,x\big)\big|^{p-2}\theta\big(t,x\big)\,\nabla \theta(t,x)\cdot \nabla (V(t, \psi(-t,x))dx\leq\\
& \frac{p-1}{2}\int_{\T^d} e^{pV(t,\psi(-t,x))}|\nabla\theta(t,x)|^2 \big|\theta\big(t,x\big)\big|^{p-2}dx\\
&+\frac{p^2}{2(p-1)}\int_{\T^d} e^{pV(t,\psi(-t,x))} \big|\theta\big(t,x\big)\big|^{p}|\nabla (V(t, \psi(-t,x))|^2dx.
\end{align*}
Consequently
\begin{align*}
\mathcal{I}(t)&\leq - \frac{p-1}{2}\int_{\T^d} e^{pV(t,\psi(-t,x))}|\nabla\theta(t,x)|^2 \big|\theta\big(t,x\big)\big|^{p-2}dx\\
&+\frac{p^2}{2(p-1)}\|\nabla (V(t, \psi(-t,\cdot))\|_{L^\infty}^2\int_{\T^d} e^{pV(t,\psi(-t,x))} \big|\theta\big(t,x\big)\big|^{p}dx.
\end{align*}
Notice that from the change of variables $x\mapsto \psi(t,x)$ one gets
$$
\int_{\T^d} e^{pV(t,\psi(-t,x))} \big|\theta\big(t,x\big)\big|^{p}dx=\|h(t)\|_{L^p}^p.
$$
From the group property of the flow we obtain
\begin{align*}
V(t,\psi(-t,x))&=\int_0^t\phi(\psi(\tau,\psi(-t,x))d\tau\\
&=\int_0^t\phi(\psi(\tau-t,x)d\tau\\
&=\int_{-t}^0\phi(\psi(\tau,x)d\tau.
\end{align*}
Applying the chain rule we infer that for any $t\in\RR$
\begin{align*}
|\nabla (V(t, \psi(\tau,\cdot))\|_{L^\infty}\le&\|\nabla \phi\|_{L^\infty}\left|\int_{-t}^0\|\nabla \psi(\tau)\|_{L^\infty}d\tau\right|.
\end{align*}
Coming back to the ODE \eqref{flow1}, then  it is straightforward to obtain that for any $t\in\RR$
$$
\|\nabla \psi(t)\|_{L^\infty}\leq e^{|t|\|\nabla v\|_{L^\infty}}.
$$
It follows that
\begin{align*}
|\nabla (V(t, \psi(\tau,\cdot))\|_{L^\infty}^2\le&\|\nabla \phi\|_{L^\infty}^2|t|e^{2|t|\|\nabla v\|_{L^\infty}}\\
\le &C_0 e^{C_0|t|}.
\end{align*}
For some constant $C_0>0$ depending only on the Lipschitz norms of  $\phi$ and $v$. 
Putting together the preceding estimates allows to get
\begin{align*}
\mathcal{I}(t)&\leq - \frac{p-1}{2}\int_{\T^d} e^{pV(t,\psi(-t,x))}|\nabla\theta (t,x)|^2 \big|\theta\big(t,x\big)\big|^{p-2}dx+C_0\frac{p^2}{p-1}e^{C_0|t|} \|h(t)\|_{L^p}^p.
\end{align*}
Inserting this into \eqref{Energ1} implies that for any $t\geq0$,
$$
\frac{1}{p}\frac{d}{dt}\| h(t)\|_{L^p}^p+\nu \frac{p-1}{2}\int_{\T^d} e^{pV(t,\psi(-t,x))}|\nabla\theta(t,x)|^2 \big|\theta\big(t,x\big)\big|^{p-2}dx\leq C_0\frac{p^2}{p-1}\nu e^{C_0t} \|h(t)\|_{L^p}^p.
$$
From  Gronwall lemma  we deduce that for any $t\geq0,$
\begin{align*}
\|h(t)\|_{L^p}^p+\nu \frac{p(p-1)}{2}\int_0^t\int_{\T^d} e^{pV(\tau,\psi(-\tau,x))}|\nabla\theta (\tau,x)|^2 \big|\theta\big(\tau,x\big)\big|^{p-2}dxd\tau\leq &e^{C_0\frac{p^3}{p-1}\nu t e^{C_0t}} \|h(0)\|_{L^p}^p\\
\leq& e^{C_0\frac{p^3}{p-1}\nu e^{C_0t}} \|\theta_0\|_{L^p}^p.
\end{align*}
By  choosing $t\geq 0$ such that $\nu e^{C_0t}\leq1,$ that is, $0\leq t\leq C_0^{-1}\ln(1/\nu)$ we obtain
\begin{align*}
\|h(t)\|_{L^p}^p+\nu \frac{p(p-1)}{2}\int_0^t\int_{\T^d} e^{pV(\tau,\psi(-\tau,x))}|\nabla\theta (\tau,x)|^2 \big|\theta\big(\tau,x\big)\big|^{p-2}dxd\tau
\leq&e^{C\frac{p^3}{p-1}}\|\theta(0)\|_{L^p}^p,
\end{align*}
In particular, we get that for any $t\in[0,C_0^{-1}\ln(1/\nu)]$ and for any $p\in[2,\infty)$, 
\begin{equation}\label{Upp-b}
\|h(t)\|_{L^p}\leq e^{C_0 p}\|\theta(0)\|_{L^p}.
\end{equation}
Using the unique ergodicity of the flows we obtain by virtue of Theorem \ref{Birk-Ergod}-(v) that
$$
\lim_{t\to+\infty}\left\|\frac1t V(t,\cdot)-\langle\phi\rangle \right\|_{L^\infty}=0.
$$
Consequently, for any $0<\mu<\langle \phi\rangle$ there exists $T_0>0$, independent of $\nu$ and $p$ such that 
$$
\forall t\geq T_0, \forall x\in\T^d,\quad \mu \leq \frac1t V(t,x).
$$
This allows to get the lower bound: for any $t\geq T_0$
\begin{align*}
\|h(t)\|_{L^p}^p=&\int_{\T^d} e^{pV(t,x)}|\theta(t,\psi(t,x))|^pdx\\
\geq &e^{p \mu  t}\|\theta(t)\|_{L^p}^p.
\end{align*}
Hence
\begin{align*}
\|h(t)\|_{L^p}
\geq &e^{ \mu t}\|\theta(t)\|_{L^p}
\end{align*}
Combing this estimate with \eqref{Upp-b} yields for any $t\in[T_0, C_0^{-1}\ln(1/\nu)]$
$$
\|\theta(t)\|_{L^p}\leq e^{C_0p-\mu t}\|\theta_0\|_{L^p}
$$
Thus we can find a constant $C_p$ depending only on $p$ and the Lipschitz norms of $\phi$ and $v$ such that 
$$
\forall t\in[0,C_0^{-1}\ln(1/\nu)]|,\quad \|\theta(t)\|_{L^p}\leq C_pe^{-\mu t}\|\theta_0\|_{L^p}.
$$
This achieves the desired result.
\end{proof}

%


\begin{thebibliography}{99}

\bibitem{Asch} M. Asch and G. Lebeau, The spectrum of the damped wave operator for a bounded domain in $\mathbb{R}^2$,
{\it Exp. Math.,} {\bf 12} (2003), 227--241.

\bibitem{Cox} S. Cox and E. Zuazua, The Rate at
which Energy Decays in a Damped String, {\it Comm. Partial
Diff. Equations.,} {\bf 19} (1994), 213--243.

\bibitem{Freitas} P. Freitas, Spectral sequences for quadratic pencils and the inverse problem for the damped wave equation, {\it J. Math. Pures et Appl.,} {\bf 78} (1999), 965--980.

\bibitem{Lebeau} G. Lebeau, Equation des ondes amorties, Algebraic and geometric methods in mathematical
physics, (Kaciveli, 1993), 73--109, Math. Phys. Stud., 19, Kluwer Acad. Publ., Dordrecht, 1996.

\bibitem{Mao} H. Maoan, Unique ergodicity of flows on the torus and the rotation number, {\it Acta Mathematica Sinica,} {\bf 4} (1988), 338--342.

\bibitem{Nikola} I. Nikolaev and E. Zhuzhoma, {\it Flows on $2-$dimensional manifolds}, Lectures notes in Mathematics, 1705, Springer-Verlag, 1999.

\bibitem{Petersen} K. Petersen, {\it Ergodic theory}, Cambridge Studies in Advanced Mathematics, vol. 2, Cambridge University Press, Cambridge, 1989.

\bibitem{Saito} T. Sait\^o, On the measure-preserving flow on the torus, {\it J. Math. Soc. Japan.,} {\bf 3} (1951), 279--284.

\bibitem{Sjostrand} J. Sj\"ostrand, Asymptotic distribution of eigenfrequencies for damped wave equations, {\it Publ. RIMS, Kyoto Univ.,} {\bf 36} (2000), 573--611.

\bibitem{Walters} P. Walters, {\it An introduction to Ergodic Theory,}, Springer-Verlag, 1982.

\end{thebibliography}
\end{document}